\documentclass[12pt]{extarticle}
\usepackage{extsizes}
\usepackage[reqno]{amsmath} 
\usepackage{amsmath}
\usepackage{amssymb}
\usepackage{amsthm}
\usepackage{amscd}
\usepackage{amsfonts}
\usepackage{graphicx}
\usepackage{dsfont}
\usepackage{fancyhdr}
\usepackage{enumerate}
\usepackage{youngtab}
\usepackage[utf8]{inputenc}
\usepackage{comment}

\usepackage{subfigure}
\usepackage[margin=1.5in]{geometry}
\usepackage{graphicx}
\usepackage{algorithm}
\usepackage{algpseudocode}
\usepackage{color}
\usepackage{hyperref}
\usepackage{notation}

\theoremstyle{theorem}
\newtheorem{corollary}{Corollary}

\newtheorem{lemma}[corollary]{Lemma}

\newtheorem{theorem}[corollary]{Theorem}
\newtheorem*{theorem*}{Theorem}

\begin{document}

\AtEndDocument{%
  \par
  \medskip
  \begin{tabular}{@{}l@{}}%
    \textsc{Gabriel Coutinho}\\
    \textsc{Dept. of Computer Science} \\ 
    \textsc{Universidade Federal de Minas Gerais, Brazil} \\
    \textit{E-mail address}: \texttt{gabriel@dcc.ufmg.br} \\ \ \\
    \textsc{Emanuel Juliano} \\
    \textsc{Dept. of Computer Science} \\ 
    \textsc{Universidade Federal de Minas Gerais, Brazil} \\
    \textit{E-mail address}: \texttt{emanuelsilva@dcc.ufmg.br}\\ \ \\
    \textsc{Thomás Jung Spier} \\
    \textsc{Dept. of Computer Science} \\ 
    \textsc{Universidade Federal de Minas Gerais, Brazil} \\
    \textit{E-mail address}: \texttt{thomasjung@dcc.ufmg.br}
  \end{tabular}}

\title{The spectrum of symmetric decorated paths}
\author{Gabriel Coutinho\footnote{gabriel@dcc.ufmg.br --- remaining affiliations in the end of the manuscript.} \and Emanuel Juliano \and Thomás Jung Spier}
\date{\today}
\maketitle
\vspace{-0.8cm}

\begin{abstract} 
    The main result of this paper states that in a rooted product of a path with rooted graphs which are disposed in a somewhat mirror-symmetric fashion, there are distinct eigenvalues supported in the end vertices of the path which are too close to each other: their difference is smaller than the square root of two in the even distance case, and smaller than one in the odd distance case. As a first application, we show that these end vertices cannot be involved in a quantum walk phenomenon known as perfect state transfer, significantly strengthening a recent result by two of the authors along with Godsil and van Bommel. For a second application, we show that there is no balanced integral tree of odd diameter bigger than three, answering a question raised by H\'{i}c and Nedela in 1998. 
    
    Our main technique involves manipulating ratios of characteristic polynomials of graphs and subgraphs into continued fractions, and exploring in detail their analytic properties. We will also make use of a result due to P\'{o}lya and Szeg\"{o} about functions that preserve the Lebesgue measure, which as far as we know is a novel application to combinatorics. In the end, we connect our machinery to a recently introduced algorithm to locate eigenvalues of trees, and with our approach we show that any graph which contains two vertices separated by a unique path that is the subdivision of a bridge with at least six inner vertices cannot be integral. As a minor corollary this implies that most trees are not integral, but we believe no one thought otherwise.
\end{abstract}

\begin{center}
\end{center}
\begin{center}
\end{center}


\section{Introduction}\label{intro}

A rooted product of a graph $H$ on $n$ vertices with rooted graphs $G_1$,...,$G_n$ is obtained upon identifying the root of each $G_i$ with the $i$th vertex of $H$. This follows the definition introduced by Godsil and Mckay \cite{GodsilMcKayRooted} which succeeded a simpler case studied by Schwenk \cite{SchwenkCharPol}. Trees are examples of graphs that can be very appropriately decomposed as rooted products. In this paper, we focus on a slightly more general case: our graph $H$ will be a path on $n$ vertices but the graphs $G_1$,...,$G_n$ are arbitrary.

\begin{figure}[H]
\begin{center}
		\includegraphics{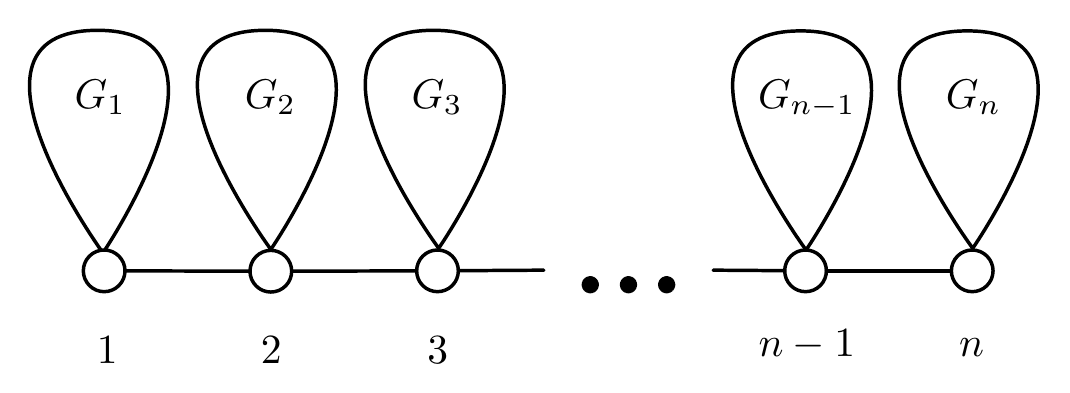}
		\caption{A decorated path $G$: the rooted product of a path $P_n$ with graphs $G_1$,...,$G_n$.}
\end{center} \label{fig:rooted}
\end{figure}

Please keep this figure in mind because we will refer back to it several times in the text. We use the following notation: a graph $G$ has a characteristic polynomial $\phi^{G}$ (all polynomials and rational functions in this text are on the variable $x$). If $i$ is a vertex in $G$, then we define the rational function
\begin{equation}
\alpha_i^G =\dfrac{\phi^G}{\phi^{G\setminus i}}.
\label{eq:alpha}
\end{equation}

It is well-known (see Section~\ref{sec:prelim} below) that $\alpha_i^G(\theta) = 0$ if and only if $\theta$ is an eigenvalue of $G$ for which there is an eigenvector with a non-zero entry at $i$.

Our main result is the following.
\begin{theorem} \label{thm:main}
	Let $G$ be a rooted graph as in Figure~\ref{fig:rooted}, and assume $G$ is neither $P_2$ nor $P_3$. Assume the following two conditions hold.
	\begin{enumerate}[(1)]
		\item For all $i \in \{1,...,n\}$,
	\[
		\alpha^{G_i}_i = \alpha^{G_{n+1-i}}_{n+1-i} \ ;
	\]
	\item For all $\theta$ so that $\alpha_1^G(\theta) = 0$, and for all vertices $i$, $\theta$ is not a pole of $\alpha_i^{G_i}$.
	\end{enumerate}
	
	\noindent
	Then there are two distinct eigenvalues $\lambda$ and $\mu$ of $A(G)$, both of which contain an eigenvector with a non-zero entry at vertex $1$, and so that $|\lambda - \mu| < \sqrt{2}$ if $n$ is odd, and $|\lambda - \mu| \leq 1$ if $n$ is even.
\end{theorem}

Condition (1) describes a sort of symmetry of walk counts going on. In fact, if it holds, the vertices $1$ and $n$ are cospectral, and so the number of closed walks of a fixed length starting at either is invariant (see Lemma~\ref{lem:alphas}). Note that condition (1) holds trivially if there is a symmetry of $G$ reflecting the path about its centre. Condition (2) arises naturally as a necessary and sufficient condition for $1$ and $n$ to be strongly cospectral, a concept originated in the study of quantum walks that has received considerable amount of attention lately (\cite{CoutinhoSpectrallyExtremal2,godsil2017strongly,kempton2020characterizing,arnadottir2021strongly,sin2022large,coutinho2022strong}). We also point out that to prove Theorem~\ref{thm:main} for when the path has even length, we had to resource to a result in Ergodic theory regarding functions that preserve the Lebesgue measure, foreseeing deeper connections yet to be explored.

Despite the apparent unnatural connection between such cryptic hypotheses and the prosaic conclusion, we argue that Theorem~\ref{thm:main} is quite relevant. First, the development of the methods we need to prove it provides a rich theory to study the connection between combinatorics and eigenvalues, and we shall see an application to the study of integer eigenvalues of trees in the end of the paper; and second, we use it to prove a strong result about continuous-time quantum walks in graphs,  generalizing recent developments in the field and moving towards the proof of a known conjecture. 

An important task in the continuous-time quantum walk model for the evolution of a qubit system is perfect state transfer: a quantum state is placed at a particular qubit and upon the constant action of a Hamiltonian for some time, this state is fully recovered with probability $1$ somewhere else in the network. In terms of the graph and its adjacency matrix, perfect state transfer is equivalent to having, for vertices $i$ and $j$ and $t \in \Rds_+$,
\[
	\exp(\ii t A)e_i = \lambda e_j,
\]
where $e_i$ and $e_j$ denote the characteristic vectors of the vertices $i$ and $j$, and $\lambda \in \Cds$, with $|\lambda| = 1$. There are plenty recent articles which discuss the basics and more of this problem (see \cite{ChristandlPSTQuantumSpinNet2} for the seminal work and \cite{CoutinhoPhD} for an introductory treatment). Most results are either purposed to show new examples of graphs admitting this phenomenon, or to rule out natural or desirable candidates. We focus on the second task: arguably quantum physicists would appreciate sparse graphs admitting state transfer between vertices at large distance, and those as in Figure~\ref{fig:rooted} with an additional mirror-symmetry are natural candidates. We use Theorem~\ref{thm:main} to show that they do not, and as a side bonus, we also show that perfect state transfer cannot occur in a simple graph at time $t = \pi$, a fact that was previously unknown.

 Additionally, we substantially generalize a result by Coutinho, Godsil, Juliano and van Bommel \cite{CoutinhoGodsilJulianovanBommel}, that showed that graphs as Figure~\ref{fig:rooted} do not admit perfect state transfer between $1$ and $n$ as long as $n = 2$ or $n = 3$ and $G_2 \cong K_1$. Our main result resembles another one by Kempton, Lippner and Yau \cite{KemptonLippnerYauPotential}, who showed that if the graphs $G_1$,...,$G_n$ in Figure~\ref{fig:rooted} are weighted loops, then perfect state transfer does not occur between vertices $1$ and $n$, answering a question asked by Godsil in \cite{GodsilStateTransfer12}.

In the end of the paper, we address two applications regarding the well studied topic of integral trees. First we show that there is no balanced integral tree of odd diameter bigger than three, answering a question raised by Híc and Nedela; and second, we display an alternative interpretation of the algorithm introduced by Jacobs and Trevisan \cite{TrevisanJacobsOriginal} to locate eigenvalues of trees. Our interpretation uses the rational functions $\alpha_i^G$, and with it we could show that, in any graph, if two vertices are connected by a unique path, obtained upon subdividing a bridge with six inner vertices, then the graph cannot be integral.

In the next section, we briefly review the main tools we need, and we explore the connection between the rational functions $\alpha_i^G$ and strong cospectrality. In Section~\ref{sec:support}, we show that a decorated path whose $\alpha$'s are symmetric about the centre of the path can be folded just as it could be if the actual graphs $G_i$ were symmetric, leading to a notion of quotient graph stronger than that obtained in the study of equitable partitions. This will allow us to have a good understanding of how the eigenvectors behave within the path. In Section~\ref{sec:main} we prove Theorem~\ref{thm:main}. In Section~\ref{sec:sym_dec_paths}, we focus on our application to quantum walks, showing the perfect state transfer cannot occur in a graph $G$ as that in Figure~\ref{fig:rooted} if $\alpha_i^G = \alpha_{n+1-i}^G$, unless $G$ is $P_2$ or $P_3$. In Section~\ref{sec:trees}, we discuss the applications to the study of integral trees.

\section{Rational functions and strong cospectrality}\label{sec:prelim}

In this paper, we will denote by $\phi^G$ the characteristic polynomial of the graph $G$ in the variable $x$. If $\theta_0,\theta_1,\dots, \theta_d$ are the distinct eigenvalues of the adjacency matrix $A$ of $G$, then we denote by $E_r$ the orthogonal projection onto the  $\theta_r$-eigenspace. 

Two vertices $i$ and $j$ of the graph $G$ are called cospectral if $\phi^{G\setminus i}=\phi^{G\setminus j}$. This is equivalent to having, for every $r$ in $\{0,1,\dots,d\}$, that $(E_r)_{i,i}=(E_r)_{j,j}$. These diagonal entries are the norm of the columns of $E_r$ corresponding to vertices $i$ and $j$. If, moreover, $E_r e_i = \pm E_r e_j$ for every $r$ in $\{0,1,\dots,d\}$, then $i$ and $j$ are said to be strongly cospectral. It is easy to verify that this is equivalent to requiring that they are cospectral and that $(E_r)_{i,i} = \pm (E_r)_{i,j}$ for every $r$ in $\{0,1,\dots,d\}$.

Fixed vertex $i$, an important distinction among eigenvalues is whether $(E_r)_{i,i}$ is zero or not. The set of eigenvalues for which this is non-zero is defined as the \textit{eigenvalue support} of vertex $i$, which we denote by $\Phi_i(G)$ or $\Phi_i$ if $G$ is defined in the context. Note that cospectral vertices must have the same eigenvalue support.

Entries in the idempotents are determined by the characteristic polynomials of vertex deleted subgraphs (see for instance \cite[Section 2.3]{CoutinhoGodsilJulianovanBommel}). In fact,

\begin{equation}(E_r)_{i,j} = \lim_{t \to \theta_r} \frac{(t-\theta_r)\sqrt{\phi^{G \setminus i}\phi^{G \setminus j} - \phi^G \phi^{G \setminus \{i,j\}}}}{\phi^ G} , \label{eq:offdiagonal}
\end{equation}
with the understanding that if $i = j$ the square root colapses to $\phi^{G \setminus i}$. This leads to the following useful characterization of strongly cospectral vertices.

\begin{theorem}[Corollary 8.4 in \cite{godsil2017strongly}]\label{thm:strcospec}
Vertices $i$ and $j$ of a graph $G$ are strongly cospectral if and only if $\phi^{G \setminus i} = \phi^{G \setminus j}$ and all poles of $\phi^{G\setminus \{i,j\}}/\phi^G$ are simple.
\end{theorem}

The expression within the square root in \eqref{eq:offdiagonal} is indeed a perfect square (when $i \neq j$), given below:

\begin{lemma}[Lemma 2.1 in \cite{GodsilAlgebraicCombinatorics}]\label{lem:wronskian} 
Let $i$ and $j$ be vertices in the graph $G$. Then,
\[\phi^{G\setminus i}\phi^{G\setminus j}-\phi^{G\setminus\{i,j\}}\phi^G=\left(\sum_{P:i\to j}\phi^{G\setminus P}\right)^2,\]
where the sum is over all the paths from $i$ to $j$.
\end{lemma}

In this paper, we will repeatedly manipulate the ratio between the characteristic polynomials of a graph and its vertex deleted subgraphs. To facilitate our notation, we introduce
\begin{equation}
\alpha_i^G =\dfrac{\phi^G}{\phi^{G\setminus i}},
\label{eq:alpha}
\end{equation}
for any graph $G$ and vertex $i$.

It is a consequence of \eqref{eq:offdiagonal} that $\alpha^G_i(\theta) = 0$ if and only if $\theta$ is an eigenvalue of $G$ in the eigenvalue support of $i$. The following lemma estabilishes analytical properties of this rational function.

\begin{lemma}[Lemma 4 in \cite{coutinho2022strong}] \label{lemma:derivative_quotient} Let $i$ be a vertex in the graph $G$. Then, $(\alpha_i^G)'(t)\geq 1$ for every $t$ that is not a pole of $\alpha_i^G$. In particular, $\alpha_i^G(t)$ has only simple zeros and poles, and is increasing and surjective on each of its branches.
\end{lemma}

Lemma~\ref{lemma:derivative_quotient} says that all zeros and all poles of $\alpha^G_i$ are simple. This is equivalent to the well-known fact that if $\theta$ is an eigenvalue of $A(G)$, then its multiplicity as an eigenvalue of $A(G\setminus i)$ differs by at most one. Precisely:
\begin{itemize}
	\item $\theta$ is a zero of $\alpha^G_i$ if and only if $\mathrm{mult}(\theta,A(G \setminus i)) + 1= \mathrm{mult}(\theta,A(G))$, and equivalently, $\theta$ is an eigenvalue of $G$ in the support of $i$;
	\item $\alpha^G_i (\theta) \neq 0,\infty$ if and only if $\mathrm{mult}(\theta,A(G \setminus i)) = \mathrm{mult}(\theta,A(G))$ (which includes the possibility both are equal to $0$);
	\item $\alpha^G_i (\theta) = \infty$ if and only if $\mathrm{mult}(\theta,A(G \setminus i)) = \mathrm{mult}(\theta,A(G)) + 1$.	
\end{itemize}

These rational functions can be used to characterize strongly cospectral vertices in graphs.

\begin{lemma}[Lemma 9 in \cite{coutinho2022strong}] \label{lemma:alpha-strcospec}
	Let $i$ and $j$ be distinct vertices in the graph $G$. Then, $i$ and $j$ are strongly cospectral if, and only if, $\alpha_i^G=\alpha_j^{G}$ and $\alpha_i^G(\theta)=\alpha_j^{G}(\theta)\neq 0$ whenever $\alpha_i^{G\setminus j}(\theta)=0$ or $\alpha_j^{G\setminus i}(\theta)=0$.
\end{lemma}

We end this section with our first results, describing a wide class of graphs such as those in Figure~\ref{fig:rooted} for which vertices $1$ and $n$ are cospectral.

\begin{lemma} \label{lem:alphas}
	Let $G$ be the rooted product of $P_n$ with graphs $G_1$,...,$G_n$, as in Figure~\ref{fig:rooted}. Assume 
	\[
		\alpha^{G_i}_i = \alpha^{G_{n+1-i}}_{n+1-i}.
	\]
	Then $\alpha_1^G=\alpha_n^{G}$.
\end{lemma}
\begin{proof}
	Schwenk showed \cite[Theorem 2]{SchwenkCharPol} that, for any vertex $a$ of $G$, if $\mathcal{C}(a)$ denotes the set of cycles containing $a$, then
	\begin{equation}
		\phi^G = x \phi^{G\setminus a} - \sum_{b \sim a} \phi^{G \setminus \{a,b\}} - 2 \sum_{C \in \mathcal{C}(a)} \phi^{G \setminus V(C)}. \label{eq:schwenk}
	\end{equation}
	Let $H$ be the subgraph of $G$ obtained upon the removal of $G_1$. Recall that the characteristic polynomial of disconnected graphs is the product of the characteristic polynomials of each component. Because all cycles of $G$ containing $i$ are entirely contained in $G_i$, it follows that
	\[
		\alpha_1^G = \alpha^{G_1}_1 - \frac{1}{\alpha_2^H},
	\]
	and therefore
	\begin{equation}
		\alpha_1^G = \alpha^{G_1}_1 - \dfrac{1}{\alpha^{G_2}_2 - \dfrac{1}{\alpha^{G_3}_3 - \dfrac{1}{\ddots\ - \dfrac{1}{\alpha^{G_n}_n}}}}, \label{eq:alphasfraction}
	\end{equation}
	from which the result follows immediately.
\end{proof}

The condition in the theorem is a natural sufficient condition, but it is not necessary. The simple example below shows two cospectral vertices for which the equalities for the $\alpha$'s do not hold.

\begin{figure}[H]
\begin{center}
		\includegraphics{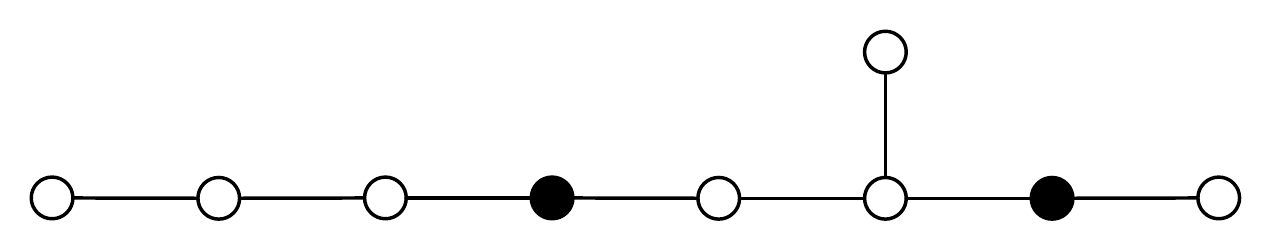}
		\caption{Shaded vertices are (strongly) cospectral, but the condition of Lemma~\ref{lem:alphas} does not apply.}
\end{center} \label{fig:schwenk}
\end{figure}

 We now assume vertices $1$ and $n$ are cospectral in $G$ (see Figure~\ref{fig:rooted}), that is, $\alpha_1^G=\alpha_n^{G}$, and show that the second condition in Lemma~\ref{lemma:alpha-strcospec} can be phrased nicely in terms of the $G_i$'s.
 
 This is a convenient moment to introduce a notation. If $f(x)$ is a rational function, we will say that $f(\theta) = \infty$ to indicate that $\theta$ is a pole of $f$. Because all rational functions in this paper are strictly increasing in each of their branches, we ask your forgiveness to state that, here, $\infty + \infty = \infty$.
 
\begin{theorem} \label{thm:strcospeccharact}
	Let $G$ be the rooted product of $P_n$ with graphs $G_1$,...,$G_n$, as in Figure~\ref{fig:rooted}, and assume $\alpha_1^G=\alpha_n^{G}$. Then $1$ and $n$ are strongly cospectral if and only if, for all $\theta$ in their eigenvalue support, and for all $k \in \{1,...,n\}$, $\alpha_k^{G_k}(\theta) \neq \infty$.
\end{theorem}
\begin{proof}
	From a simple expansion on each side, note that $\alpha_1^G\alpha_n^{G \setminus 1} = \alpha_n^G\alpha_1^{G \setminus n}$, and because $\alpha_1^G=\alpha_n^{G}$ by hypothesis, it follows that $\alpha_n^{G \setminus 1} = \alpha_1^{G \setminus n}$. From Equation~\eqref{eq:alphasfraction}, it follows that
	\[
	\alpha_1^G = \alpha^{G_1}_1 - \dfrac{1}{\alpha^{G_2}_2 - \dfrac{1}{\ddots\ -\dfrac{1}{\alpha^{G_n}_n}}},  \quad \text{and} \quad \alpha_1^{G\setminus n} = \alpha^{G_1}_1 - \dfrac{1}{\alpha^{G_2}_2 - \dfrac{1}{\ddots\ - \dfrac{1}{\alpha^{G_{n-1}}_{n-1}}}}.
	\]

	From Lemma~\ref{lemma:alpha-strcospec}, the two vertices are strongly cospectral if, and only if, $\alpha_1^G(\theta) = 0$ implies that $\alpha_1^{G \setminus n} (\theta) \neq 0$, that is, if for all $\theta$ in the eigenvalue support of $1$, $\alpha_1^{G \setminus n} (\theta) \neq 0$.
	
	To show one direction, assume that for some $k$ and some $\theta$ in the eigenvalue support, we have $\alpha_k^{G_k}(\theta) = \infty$. If $H$ is obtained from $G$ upon removal of vertices $1$ to $k$, note that
	\[
	\alpha_1^H = \alpha^{G_{k+1}}_{k+1} - \dfrac{1}{\alpha^{G_{k+2}}_{k+2} - \dfrac{1}{\ddots\ -\dfrac{1}{\alpha^{G_n}_n}}},  \quad \text{and} \quad \alpha_1^{H\setminus n} = \alpha^{G_{k+1}}_{k+1} - \dfrac{1}{\alpha^{G_{k+2}}_{k+2} - \dfrac{1}{\ddots\ - \dfrac{1}{\alpha^{G_{n-1}}_{n-1}}}}.
	\]
	Both these functions have simple poles and simple zeros and are strictly increasing in their branches (Lemma~\ref{lemma:derivative_quotient}), thus $-(\alpha_1^H)^{-1}$ and $-(\alpha_1^{H\setminus n})^{-1}$ also have simple zeros and simple poles and are strictly increasing in their branches, therefore for all $\theta$ which is a pole of $\alpha_k^{G_k}$, even if $\theta$ is one of the poles of $-(\alpha_1^H)^{-1}$ or $-(\alpha_1^{H\setminus n})^{-1}$, it still follows that
	\[
		\alpha_k^{G_k}(\theta) - \frac{1}{\alpha_1^H(\theta)} = \alpha_k^{G_k}(\theta) - \frac{1}{\alpha_1^{H\setminus n}(\theta)} = \infty.
	\]
	(that is, $\infty + \infty = \infty$.) As a consequence, $\alpha_1^G(\theta) = \alpha_1^{G\setminus n}(\theta)$, and the vertices cannot be strongly cospectral.
	
	For the other direction, assume that for all $\theta$ in the eigenvalue support of $1$ and $n$, and for all $k \in \{1,...,n\}$, $\alpha_k^{G_k}(\theta) \neq \infty$. Let $G(k)$ denote the graph $G$ with vertices $\{1,...,k-1\}$ removed. Assume for the sake of finding a contradiction that there exists $\theta$ giving $\alpha_1^G(\theta) = \alpha_1^{G \setminus n} (\theta) = 0$. This implies that
	\[
		\alpha_2^{G(2)}(\theta) = \alpha_2^{G(2) \setminus n} (\theta).
	\]
	By induction now, if $\alpha_k^{G(k)}(\theta) = \alpha_k^{G(k) \setminus n} (\theta)$, and because $\alpha_k^{G_k}(\theta) \neq \infty$, it must be that $\alpha_{k+1}^{G(k+1)}(\theta) = \alpha_{k+1}^{G(k+1) \setminus n} (\theta)$. Thus we conclude that
	\[
		\alpha^{G_{n-1}}_{n-1}(\theta) = \alpha^{G_{n-1}}_{n-1}(\theta) - \frac{1}{\alpha^{G_{n}}_{n}(\theta)},
	\]
	which cannot be true, as $\theta$ is not a pole of any of these terms.
\end{proof}

Theorem~\ref{thm:strcospeccharact} simultaneously generalizes both constructions of strongly cospectral vertices presented in Theorem 9.1 and Lemma 9.3 of \cite{godsil2017strongly}.

We finish this section with a short lemma.

\begin{lemma} \label{lem:short}
	Let $G$ be a graph as in Figure~\ref{fig:rooted}. If $\theta$ is larger than the largest pole of $\alpha_1^G$, then $\alpha_i^{G_i}(\lambda) \neq \infty$ for all $i$ and all $\lambda \geq \theta$.
\end{lemma}
\begin{proof}
	Consider the graphs $G(k)$ as obtained upon removing $G_1,...,G_{k-1}$ from $G$. If the statement is false, then there is at least one $k$ so that $\alpha_{k}^{G(k)}(\lambda) = \infty$. Let $i$ be smallest such case. Note that $i \neq 1$ because $G(1) = G$. As $\alpha_{i}^{G(i)}(\lambda) = \infty$, at a larger value $\lambda'$ we have $\alpha_{i}^{G(i)}(\lambda') = 0$, and therefore
		\[
			\alpha_{i-1}^{G(i-1)}(\lambda') = \alpha_{i-1}^{G_{i-1}}(\lambda') - \frac{1}{\alpha_{i}^{G(i)}(\lambda')} = \infty,
		\]
		as both terms on the right are strictly increasing in their branches, leading to a contradiction.
\end{proof}

\section{Folding the symmetric decorated paths} \label{sec:support}

We assume vertices $1$ and $n$ are strongly cospectral, that is, for all eigenvalues $\theta$ in their eigenvalue support, if $E_\theta$ is the projection onto the eigenspace, then $E_\theta e_1 = \pm E_\theta e_n$. We will use $\Phi_{1n}^+(G)$ to denote the set of eigenvalues of $A(G)$ in their support for which $E_\theta e_1 = E_\theta e_n$, and have $\Phi_{1n}^-(G)$ analogously defined. We would like to describe these sets in terms of smaller graphs that, in a sense, behave like quotients. It will be necessary to distinguish between the cases where the number of vertices in the path is odd or even.

\begin{figure}[H]
\begin{center}
		\includegraphics{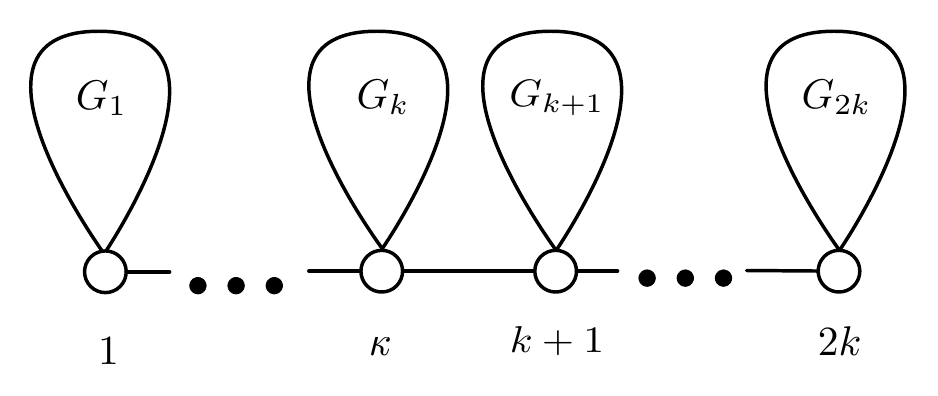}
		\caption{Rooted product $G$ of a path $P_{2k}$ with graphs $G_1$,...,$G_{2k}$.} \label{fig:rooted2n}
\end{center} 
\end{figure}

\begin{figure}[H]
\begin{center}
		\includegraphics{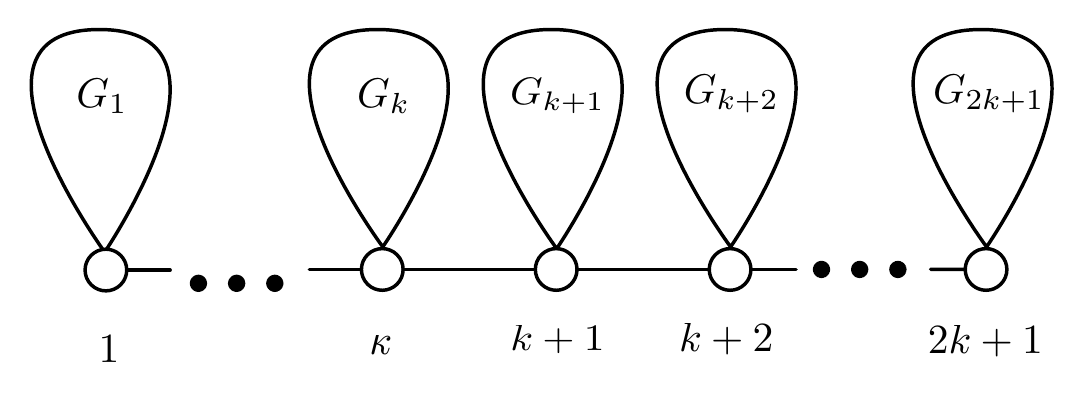}
		\caption{Rooted product $G$ of a path $P_{2k+1}$ with graphs $G_1$,...,$G_{2k+1}$.} \label{fig:rooted2n1}
\end{center} 
\end{figure}

If $G_i = G_{n+1-i}$, then there is an automorphism of $G$ that swaps these pairs, and the orbits of the subgroup generated by this reflection form an equitable partition. Thus, any eigenvector of $G$ is either constant on the classes, or orthogonal to their characteristic vectors (see \cite[Section 9.3]{GodsilRoyle}). Hence there are three types of eigenvectors: those constant and nonzero at $1$ and $n$, those of opposing sign, and those equal to zero. Assuming $1$ and $n$ are strongly cospectral, we can therefore describe the sets $\Phi_{1n}^\pm(G)$ as the eigenvalues in the support $1$ of the following graphs: 

\begin{figure}[H]
\begin{center}
		\includegraphics{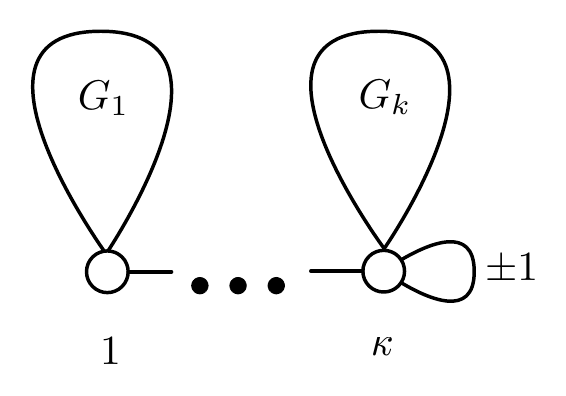}
		\caption{Graphs $G^{\pm}$, where the $\pm 1$ are weights in the loop. The eigenvalues in the support of $1$ are the eigenvalues in $\Phi_{1n}^\pm(G)$ for when $G$ is a rooted product over $P_n = P_{2k}$.} \label{fig:rooted2nsplits}
\end{center} 
\end{figure}

\begin{figure}[H]
\begin{center}
		\includegraphics{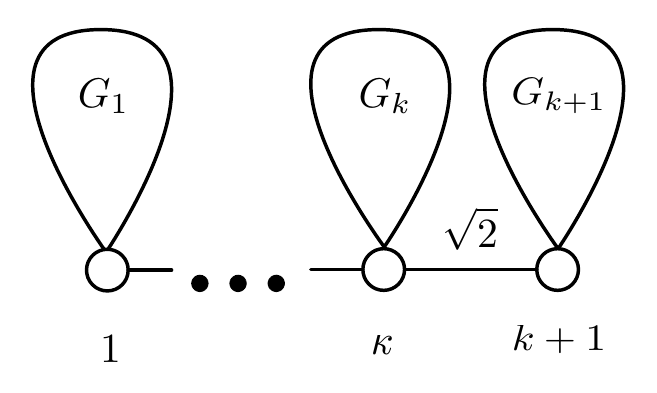} \quad
		\includegraphics{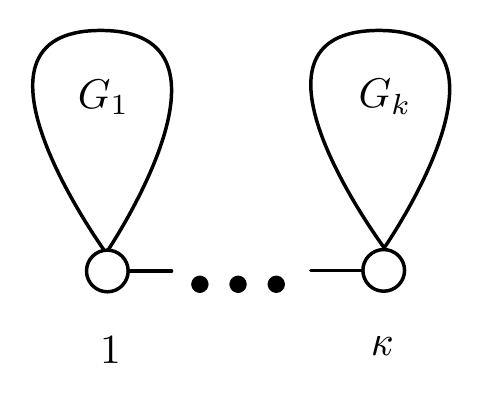}
		\caption{Graphs $G^{\pm}$ respectively. On $G^+$, the $\sqrt{2}$ is the weight of the edge. The eigenvalues in the support of $1$ are the eigenvalues in $\Phi_{1n}^\pm(G)$ for when $G$ is a rooted product over $P_n = P_{2k+1}$.} \label{fig:rooted2n1splits}
\end{center} 
\end{figure}

The main result in this section is that the hypothesis $G_i = G_{n+1-i}$ can be relaxed to $\alpha^{G_i}_i = \alpha^{G_{n+1-i}}_{n+1-i}$. Of course we can no longer directly use a combinatorial equitable partition. This result was obtained in \cite{CoutinhoGodsilJulianovanBommel} for the paths on $2$ or $3$ vertices, but here we use a different method to prove the general case. First we define a construction of a path with real weighted loops, based on a graph $G$ that is a rooted product on $P_n$:

\begin{figure}[H]
\begin{center}
		\includegraphics{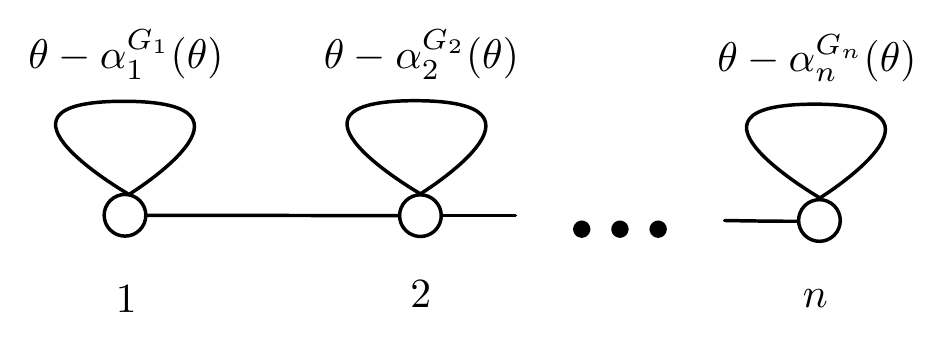}
		\caption{Given $\theta \in \Rds$, this is the graph $G[\theta]$ obtained from $G$ upon replacing the graphs $G_i$ by conveniently real weighted loops, as long as each $\alpha_i^{G_i}(\theta) \neq \infty$.} \label{fig:looped}
\end{center} 
\end{figure}

\begin{lemma} \label{lem:rootedtopath}
	Assume $G$ is as in Figures~\ref{fig:rooted2n} or \ref{fig:rooted2n1}. Assume that $1$ and $n$ are strongly cospectral. Then $\theta$ is an eigenvalue of $G$ in the support of $1$ if and only if $\theta$ is an eigenvalue in the support of $1$ in the graph $G[\theta]$.
\end{lemma}

\begin{proof}
	The first thing to notice is that $H = G[\theta]$ is well-defined: the fact that $1$ and $n$ are strongly cospectral in $H$ implies that $\alpha_i^{G_i}(\theta) \neq \infty$ (Lemma~\ref{thm:strcospeccharact}). A number $\tau$ is an eigenvalue of $H$ in the support of $1$ if and only if $\alpha_1^H(\tau) = 0$, and $\alpha_1^H$ admits an expansion just like Equation~\eqref{eq:alphasfraction}, in which each $H_i$ is just a vertex with the loop attached. Noticing that $\alpha_i^{H_i}(x) = x-(\theta - \alpha_i^{G_i}(\theta))$, it follows that $\alpha_1^{G}(\theta) = 0$ if and only if $\alpha_1^H(\theta) = 0$.
\end{proof}

Vertices $1$ and $n$ in the paths described in Figure~\ref{fig:looped} are strongly cospectral as the path is mirror-symmetric (see \cite{VinetZhedanovHowTo}), and their support is the entire spectrum of the path. Once an eigenvector entry at $1$ (or $n$) is determined, the rest of the eigenvector follows, so it is easy to see the following lemma.

\begin{lemma} \label{lem:evencase}
	Let $G$ be as in Figure~\ref{fig:rooted2n}, and $G^\pm$ as in Figure~\ref{fig:rooted2nsplits}. Assume that $\alpha^{G_i}_i = \alpha^{G_{n+1-i}}_{n+1-i}$, and that $1$ and $n$ are strongly cospectral. The following are equivalent (choosing either $+$ or $-$ in each):
	\begin{enumerate}[(i)]
		\item The number $\theta$ belongs to $\Phi_{1n}^\pm(G)$.
		\item $G[\theta]$ is well defined and $\theta \in \Phi_{1n}^\pm(G[\theta])$.
		\item $G[\theta]$ is well defined and $\theta$ is an eigenvalue of $G^\pm[\theta]$ in the support of $1$. (in the definition of $G^\pm[\theta]$, add the weights of the two loops in vertex $k$.)
		\item $\theta$ is an eigenvalue of $G^\pm$ in the support of $1$. (equivalently, $\alpha_1^{G^\pm}(\theta) = 0$.)
	\end{enumerate}
\end{lemma}

\begin{proof}
	Note that in principle we can only say from Lemma~\ref{lem:rootedtopath} that if $\theta$ is in the support of $1$ in $G$, then it is in the support of $1$ in $G[\theta]$. We will see in the end that the sign in $\Phi_{1n}^\pm(G)$ will correspond to the sign in $\theta \in \Phi_{1n}^\pm(G[\theta])$. The equivalence between (ii) and (iii) follows straight from considering the that the eigenvectors of $G[\theta]$ are symmetric or anti-symmetric about the centre of the path. Lemma~\ref{lem:rootedtopath} gives that (iii) $\implies$ (iv). If we assume (iv), we can use the symmetry of the $\alpha$'s and Equation~\eqref{eq:alphasfraction} to replace $G_i$ by $G_{n+1-i}$ in the definitions of $G^\pm$, and so (iv) implies (i) because we can easily construct the eigenvectors of $G$ by gluing eigenvectors for $G^\pm$ and for their modified versions, and from this construction, it follows that $\theta \in \Phi_{1n}^\pm(G)$ implies $\theta \in \Phi_{1n}^\pm(G[\theta])$, giving that (i) $\implies$ (ii) with the corresponding signs.
\end{proof}

For reference, we also state the lemma for the odd case, but the proof is analogous.

\begin{lemma} \label{lem:oddcase}
	Let $G$ be as in Figure~\ref{fig:rooted2n1}, and $G^\pm$ as in Figure~\ref{fig:rooted2n1splits}. Assume that $\alpha^{G_i}_i = \alpha^{G_{n+1-i}}_{n+1-i}$, and that $1$ and $n$ are strongly cospectral. The following are equivalent (choosing either $+$ or $-$ in each):
	\begin{enumerate}[(i)]
		\item The number $\theta$ belongs to $\Phi_{1n}^\pm(G)$.
		\item $G[\theta]$ is well defined and $\theta \in \Phi_{1n}^\pm(G[\theta])$.
		\item $G[\theta]$ is well defined and $\theta$ is an eigenvalue of $G^\pm[\theta]$ in the support of $1$. (in the definition of $G^\pm[\theta]$, add the weights of the two loops in vertex $k$.)
		\item $\theta$ is an eigenvalue of $G^\pm$ in the support of $1$. (equivalently, $\alpha_1^{G^\pm}(\theta) = 0$.)
	\end{enumerate}
\end{lemma} \qed

\section{Spectrum of symmetric decorated paths} \label{sec:main}

Our goal in this section is to prove Theorem~\ref{thm:main}. The following is our main lemma for the odd length case.

\begin{lemma} \label{lem:main_lemma}
Consider functions $\alpha(x)$ and $\beta(x)$ both of the form 
\[
x - \tau_0 - \sum_m \frac{\lambda_m}{x - \tau_m},
\] 
with $\tau_m$ and $\lambda_m$ real numbers given by the choice of the function, $\lambda_m > 0$. Assume either $\alpha$ or $\beta$ have at least one pole. Let $\lambda > 0$. Given $\theta \in \Rds$, if $\beta(\theta) = 0$ and $\alpha(\theta) \neq \infty$, then exists $\theta' \in (\theta - \sqrt{\lambda}, \theta + \sqrt{\lambda})$ such that $\beta(\theta') \neq 0, \infty$ and $\alpha(\theta') - \frac{\lambda}{\beta(\theta')} = 0$.
\end{lemma}
\begin{proof}
    Consider $f(x) := \alpha(x)\beta(x) - \lambda$. As $\alpha(\theta) \neq \infty$ and $\beta(\theta) = 0$, we have that $f(\theta) = - \lambda$.
    
    Among the poles of $\alpha$ or $\beta$, let $\theta + d_+$ and $\theta - d_-$ be the first poles after and before $\theta$, respectively. It might be the case that $\theta + d_+ = \infty$ or $\theta - d_- = - \infty$. Note that $\alpha$, $\beta$ and $f$ are continuous at $(\theta - d_0, \theta + d_+)$. It is straightforward to check that $\alpha', \beta' \geq 1$ at $(\theta - d_-, \theta + d_+)$, implying that
\[
\forall \epsilon \in (0, d_+), \text{ we have} \quad 
\begin{cases}
\alpha(\theta + \epsilon) \geq \alpha(\theta) + \epsilon, \\
\beta(\theta + \epsilon) \geq \beta(\theta) + \epsilon = \epsilon;
\end{cases}
\]
and
\[
\forall \epsilon \in (0, d_-), \text{ we have} \quad \begin{cases}
\alpha(\theta - \epsilon) \leq \alpha(\theta) - \epsilon,  \\
\beta(\theta - \epsilon) \leq \beta(\theta) - \epsilon = - \epsilon.
\end{cases}
\]

As either $\alpha(x)$ or $\beta(x)$ has at least one pole, the derivative of at least one of them will be strictly greater than $1$, so at least one of the inequalities in each pair is strict. We have two cases:

\begin{enumerate}[(i)]
	\item If $\alpha(\theta) \geq 0$, then, $\forall \epsilon \in (0, d_+)$,
	\[
	f(\theta + \epsilon) > (\alpha(\theta) + \epsilon) \epsilon - \lambda \geq \epsilon^2 - \lambda.
	\]
	\item If $\alpha(\theta) \leq 0$, then, $\forall \epsilon \in (0, d_-)$,
	\[
	f(\theta - \epsilon) > (-\alpha(\theta) + \epsilon) \epsilon - \lambda \geq \epsilon^2 - \lambda
	\]
\end{enumerate}

We also note that $\lim_{\epsilon \to d_+} f(\theta + \epsilon) > 0$ and $\lim_{\epsilon \to d_-} f(\theta - \epsilon) > 0$.

Therefore, in either case, and because $f(\theta) < 0$, if $\mu = \min\{\epsilon,\sqrt{\lambda}\}$, we find that there is $\theta' \in (\theta - \mu, \theta + \mu)$ with $f(\theta') = 0$ and $\beta(\theta') \neq \infty$. Because the roots and poles of $\beta$ intercalate, and $\beta(\theta) = 0$, we also conclude that $\beta(\theta') \neq 0$.
\end{proof}

We restate Theorem~\ref{thm:main} below for your convenience, followed by its proof.

\begin{theorem*}
	Let $G$ be a rooted graph as in Figure~\ref{fig:rooted}, and assume $G$ is neither $P_2$ nor $P_3$. Assume the following two conditions hold.
	\begin{enumerate}[(1)]
		\item For all $i \in \{1,...,n\}$,
	\[
		\alpha^{G_i}_i = \alpha^{G_{n+1-i}}_{n+1-i} \ ;
	\]
	\item For all $\theta$ so that $\alpha_1^G(\theta) = 0$, and for all vertices $i$, we have $\alpha_i^{G_i}(\theta) \neq\infty$.
	\end{enumerate}
	
	\noindent
	Then there are two distinct eigenvalues $\lambda$ and $\mu$ of $A(G)$, both of which contain an eigenvector with a non-zero entry at vertex $1$, and so that $|\lambda - \mu| < \sqrt{2}$ if $n$ is odd, and $|\lambda - \mu| \leq 1$ if $n$ is even.
\end{theorem*}

\begin{proof}
	We split the proof into two cases depending on the parity of the length of the path.

\begin{center}
	\textsc{odd length}	
\end{center}
	
	Recall the definitions of $G^\pm$ from Figure~\ref{fig:rooted2n1splits}. We know from Lemma~\ref{lem:oddcase} that $\Phi_{1n}^\pm$ are the zeros of $\alpha_1^{G^\pm}$.
	
	Let $\alpha = \alpha_{k+1}^{G_{k+1}}$ and $\beta = \alpha_{k}^{G^-}$. We want to apply Lemma~\ref{lem:main_lemma} for these choices. It is immediate to check they are in desired form, and we first assume either $\alpha$ or $\beta$ has at least one pole. 
	
	Let $\theta$ be the largest zero of $\beta$, and thus also a zero of $\alpha_{1}^{G^-}$, therefore belonging to $\Phi_{1n}^-$. Note that $\alpha(\theta) \neq \infty$ because the vertices $1$ and $n$ are strongly cospectral (Theorem~\ref{thm:strcospeccharact}). Thus, choosing $\lambda = 2$, we find that there is $\theta' \in (\theta-\sqrt{2},\theta+\sqrt{2})$ such that $\beta(\theta') \neq 0,\infty$, and $\alpha(\theta')-2/\beta(\theta') = 0$. Similarly to our derivation of Equation~\eqref{eq:alphasfraction}, it follows from Equation~\eqref{eq:schwenk} that
	\[
		\alpha - \frac{2}{\beta} = \alpha_{k+1}^{G^+}.
	\] 
	Thus $\alpha_{k+1}^{G^+}(\theta') = 0$, but we wanted $\alpha_{1}^{G^+}(\theta') = 0$. To see this, note that as $\theta$ is the largest zero of $\beta$, then $\theta'$ is larger than the largest pole of $\beta$. This is because $\beta$ has no pole to the right of $\theta$ and the lemma guarantees $\theta'$ is to the right of the closest pole to the left of $\theta$. Thus, $\alpha_i^{G_i}(\theta') \neq\infty$ for all $i \in \{1,...,k\}$ (by Lemma~\ref{lem:short}), and also $\alpha_{k+1}^{G^{k+1}}(\theta') \neq \infty$ otherwise $\alpha_{k+1}^{G^+}(\theta') = \infty$. This is actually enough for us to apply the correspondence described in Lemma~\ref{lem:rootedtopath}, and as the extreme vertices of paths contain all eigenvalues in their support, it follows that $\alpha_{1}^{G^+}(\theta') = 0$.
	
	The case remaining is when both $\alpha$ and $\beta$ have no poles, which corresponds precisely to $P_3$. 
	
	\begin{center}
	\textsc{even length}	
	\end{center}
	
	Here we do not need Lemma~\ref{lem:main_lemma}, but the proof is not elementary. Let $G$ be as in Figure~\ref{fig:rooted2n}, and $G^\pm$ as in Figure~\ref{fig:rooted2nsplits}.
	
	From Lemma~\ref{lem:evencase}, $\theta \in \Phi_{1n}^\pm$ if and only if $\alpha_{1}^{G^\pm}(\theta) = 0$. However what we know is that $\alpha_k^{G^-} - 2 = \alpha_k^{G^+}$. So we would like to find zeros of $\alpha_k^{G^\pm}$ nearby which are also zeros of $\alpha_{1}^{G^\pm}$. First note that unless $k = 1$ and $G_1 = K_1$ (and $G$ is $P_2$), then $\alpha_k^{G^+}$ has at least two zeros. 
	
	We proceed as in the previous case, selecting $\theta$ the largest zero of $\alpha_k^{G^+}$. From Lemma~\ref{lem:short}, $\alpha_i^{G_i}(\theta) \neq \infty$, and thus the correspondence of Lemma~\ref{lem:rootedtopath} gives that $\theta$ is also a zero of $\alpha_1^{G^+}$. Note that, $\alpha_k^{G^+}$ and $\alpha_k^{G^-}$ have the same poles, so the closest zero of $\alpha_k^{G^-}$ to $\theta$, say $\theta'$, is larger than all poles of $\alpha_k^{G^-}$. So $\alpha_i^{G_i}(\theta') \neq \infty$, and we conclude that $\alpha_1^{G^-}(\theta') = 0$.
	
	If we choose $\theta$ the smallest zero of $\alpha_k^{G^+}$, then Lemma~\ref{lem:short} and Lemma~\ref{lem:rootedtopath} apply again, and the fact that the derivative is $\geq 1$ gives that there is $\theta' \in (\theta-2,\theta)$ with $\alpha_1^{G^+}(\theta) = 0$ and $\alpha_1^{G^-}(\theta') = 0$
	
	From a result due to Pólya and Szegö about functions that preserve the Lebesgue measure (see the unique unnumbered theorem in the paper \cite{letac1977functions} by Letac), if $\alpha$ is a function of the form described in Lemma~\ref{lem:main_lemma}, and if $\mu$ is the Lebesgue measure, then for any measurable subset $E$, it follows that
	\[
		\mu(E) = \mu(\alpha^{-1}(E)).
	\] 
	So we take $\theta_1$ and $\theta_2$ the smallest and largest zeros of $\alpha_1^{G^+}$ (and also $\alpha_k^{G^+}$), and we know that there are $\theta_1' \in (\theta_1-2,\theta_1)$ and $\theta_2' \in (\theta_2-2,\theta_2)$ which are the smallest and largest zeros of $\alpha_1^{G^-}$ (and also $\alpha_k^{G^-}$). Considering now the function
	\[
		\alpha = \alpha_k^{G^+} + 1 = \alpha_k^{G^-} - 1,
	\]
	and the set $E = (-1,1)$, the theorem in \cite{letac1977functions} gives that
	\[
		2 = \mu(E) = \mu(\alpha^{-1}(E)) \geq \mu((\theta_1',\theta_1)) + \mu((\theta_2',\theta_2)) = (\theta_1 -\theta_1') +  (\theta_2 -\theta_2'),
	\]
	therefore either $\theta_1' \in [\theta_1-1,\theta_1)$ or $\theta_2' \in [\theta_2-1,\theta_2)$.
	\end{proof}

\section{No perfect state transfer in symmetric decorated paths}\label{sec:sym_dec_paths}

One of the questions that motivated this work is whether graphs of the form described in Figure~\ref{fig:rooted} are candidates to admit perfect state transfer. Godsil showed that two conditions for perfect state transfer are necessary: the vertices have to be strongly cospectral \cite{GodsilStateTransfer12} ; and the vertices have to be periodic, and in this case their eigenvalue support consists of integers or quadratic integers of a particular form \cite{GodsilPerfectStateTransfer12}. An immediate consequence of the latter fact is that the closest distinct eigenvalues in the eigenvalue support of a periodic vertex are either at distance $1$ or at least $\sqrt{2}$ apart. The full characterization of perfect state transfer is stated in \cite[Chapter 2]{CoutinhoPhD}, and we reproduce it below for your convenience.

\begin{theorem} \label{thm:pstcha}
	Let $G$ be a graph, $A(G) = \sum \theta_r E_r$ the spectral decomposition of its adjacency matrix, and let $i,j \in V(G)$. There is perfect state transfer between $i$ and $j$ at time $t$ if and only if the following conditions hold.
	\begin{enumerate}[(a)]
		\item $E_r e_i = \sigma_r E_r e_j$, with $\sigma_r \in \{-1,+1\}$ (that is, vertices $i$ and $j$ are strongly cospectral.)
		\item There is an integer $a$, a square-free positive integer $\Delta$ (possibly equal to 1), so that for all $\theta_r$ in the support of $i$, there is $b_r$ giving
			\[
				\theta_r = \frac{a + b_r \sqrt{\Delta}}{2}.
			\]
			In particular, because $\theta_r$ is an algebraic integer, it follows that all $b_r$ have the same parity as $a$.
		\item There is $g \in \Zds$ so that, for all $\theta_r$ in the support of $i$, $(b_0 - b_r)/g = k_r$, with $k_r \in \Zds$, and ${k_r \equiv (1-\sigma_r)/2 \pmod 2}$.
	\end{enumerate}
	If the conditions hold, then the positive values of $t$ for which perfect state transfer occurs are precisely the odd multiples of $\pi/(g \sqrt{\Delta})$.
\end{theorem}

Most results denying perfect state transfer in families of graphs either show that candidate vertices are not strongly cospectral; or, independently, that eigenvalues are not of the desired form. We were quite pleased to find out that for graphs of the form depicted in Figure~\ref{fig:rooted}, it is strong cospectrality that allows us to show that the eigenvalues are too near each other. Note however that there is still one piece missing: Theorems \ref{thm:main} and \ref{thm:pstcha} guarantee that if $G$ is just like in Figure~\ref{fig:rooted} with mirror-symmetric $\alpha$'s, $G$ is neither $P_2$ nor $P_3$, and perfect state transfer occurs in $G$, then $g = \Delta = 1$, and equivalently, $\pi$ is the minimum time the transfer occurs. We show below that this is never the case, for any (simple unweighted)  graph, with a proof that is independent from the rest of this paper.

\begin{theorem}
	Assume perfect state transfer occurs in $G$ at minimum time $t$. Then $t \leq \pi/\sqrt{2}$, and therefore either $g$ or $\Delta$ (as in Theorem~\ref{thm:pstcha}) are $\geq 2$. 
\end{theorem}

\begin{proof}
	Assume state transfer occurs between $i$ and $j$ at time $\pi$, and thus $g = \Delta = 1$. Recall the definition of $\Phi_{ij}^\pm$, and note that these correspond respectively to the indices $r$ with $\sigma_r = \pm 1$ in Theorem~\ref{thm:pstcha}. It follows from condition (c) in said theorem that 
	\begin{enumerate}[(i)]
		\item either $\Phi_{ij}^+ \subseteq 2 \Zds$ and $\Phi_{ij}^- \subseteq 2 \Zds + 1$;
		\item or $\Phi_{ij}^+ \subseteq 2 \Zds+1$ and $\Phi_{ij}^- \subseteq 2 \Zds$.
	\end{enumerate}
    We show below that this parity separation is impossible. 
    
    Fixed $\theta$, note that $E_\theta e_i = \pm E_\theta e_j$ if and only if $(E_\theta)_{ii} = (E_\theta)_{jj} = \pm (E_\theta)_{ij}$ (this is easy, but see for instance \cite{godsil2017strongly}). From Equation~\eqref{eq:offdiagonal} and Lemma~\ref{lem:wronskian}, we have
    \[
    \frac{\phi^{G \setminus i}}{\phi^G} = \sum_{\theta} \frac{(E_{\theta})_{ii}}{x - \theta} \quad \text{and} \quad \frac{\sum_{P:i\to j}\phi^{G \setminus P}}{\phi^G} = \sum_{\theta} \frac{(E_{\theta})_{ij}}{x - \theta}.
    \]
    This implies that $\Phi_{ij}^+$ and $\Phi_{ij}^-$ are respectively the set of poles of 
    \[
        \frac{\phi^{G \setminus i} + \sum_{P : i \to j} \phi^{G \setminus P}}{\phi^G} \quad \text{and} \quad \frac{\phi^{G \setminus i} - \sum_{P : i \to j} \phi^{G \setminus P}}{\phi^G}.
    \]
    
    We write 
    \[\frac{\phi^{G \setminus i} + \sum_{P : i \to j} \phi^{G \setminus P}}{\phi^G} = \frac{p^+}{q^+},\]
    where $p^+, q^+ \in \mathds{Z}[x]$ are monic, $\gcd(p^+, q^+) = 1$ and $\text{deg } q^+ = \text{deg } p^+ + 1$. Consider $p^-$ and $q^-$ analogously. Then
    \[
        \begin{cases}
            (\phi^{G \setminus i} + \sum_{P : i \to j} \phi^{G \setminus P}) q^+ = \phi^G p^+ \ , \text{and} \\
            (\phi^{G \setminus i} - \sum_{P : i \to j} \phi^{G \setminus P}) q^- = \phi^G p^-,
        \end{cases}
    \]
    implying that
    \[
        \phi^G (p^+q^- - p^-q^+) \equiv 0 \text{ in } \mathds{F}_2[x] \implies p^+ q^- \equiv p^- q^+ \text{ in } \mathds{F}_2[x]
    \]

    This last implication arriving from the fact that $\phi^G$ is monic, and so ${\phi^G \not\equiv 0 \text{ in } \mathds{F}_2[x]}$. As $\text{deg } q^+ = \text{ deg } p^+ + 1$, it follows that the $\gcd (q^-, q^+) \not\equiv 1 \text{ in } \mathds{F}_2[x]$. However, since $\Phi_{ij}^+$ and $\Phi_{ij}^-$ are disjoint in $\Zds_2$, and these are respectively the zeros of $q^+$ and $q^-$, it must be that $\gcd(q^+, q^-) = 1 \text{ in } \mathds{F}_2[x]$, a contradiction.
\end{proof}

\begin{corollary}
	Let $G$ be a graph of the form 
\begin{figure}[H]
\begin{center}
		\includegraphics{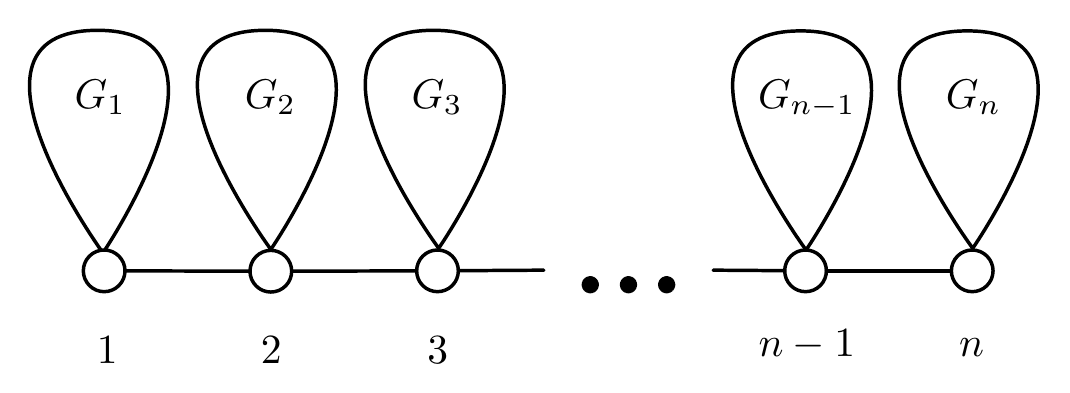}
\end{center} 
\end{figure}
\noindent with $\alpha_i^{G_i} = \alpha_{n+1-i}^{G_{n+1-i}}$. Assume $G$ is neither $P_2$ nor $P_3$. Then, there cannot be perfect state transfer between $1$ and $n$.
\end{corollary}

This corollary rules out most natural candidates to admit perfect state transfer among trees. It has been conjectured in \cite{CoutinhoLiu2} that no tree other than $P_2$ and $P_3$ admit perfect state transfer, so the only cases remaining are trees whose strongly cospectral vertices are like those depicted in Figure~\ref{fig:schwenk}. The corollary also immediately generalizes the result in \cite{CoutinhoGodsilJulianovanBommel}, which proved the special case for then $n=2$ or $n=3$ and $G_2 = K_1$, using a completely different approach with variational methods for eigenvalues.

\section{Integral trees} \label{sec:trees}

A major topic of interest in spectral graph theory has always been the pursuit of graphs with integral spectrum (see \cite{harary1974graphs}). Trees have particularly received attention lately: to cite two recent and important references, Brouwer \cite{brouwer2008small} characterized all integral trees on at most 50 vertices and Csikv{\'a}ri \cite{csikvari2010integral} showed that there are integral trees with arbitrarily large diameter (see \cite{brouwer2011integral,brouwer2008integral,ghorbani2016integral} for other recent works). H\'{i}c and Nedela \cite{hic1998balanced} introduced balanced trees as candidates to have integral spectrum. Given a tree $T$, its centre is either the vertex or the edge that sits right in the middle of a diametral path (it is a standard exercise to show that any diametral path yields the same centre). Of course the centre of trees of even diameter is a vertex, and the centre of those of odd diameter is an edge. A tree is called balanced if all vertices at the same distance from the centre of the tree have the same degree. Any balanced tree is completely determined by its diameter parity and the sequence of degrees of the vertices determined by their distance to the centre. H\'{i}c and Nedela constructed several examples of balanced integral trees with even diameter and with diameter three. They showed that there are no balanced integral tree of diameters $7$ or $4k+1$ for any $k>0$, and left it open the question regarding the remaining odd cases. Our work in Section~\ref{sec:main} leads to a direct answer in the negative: there are no balanced integral trees of diameter greater than three.

\begin{corollary}
	Let $T$ be a tree obtained by the rooted product of trees $T_1$ and $T_2$ over $P_2$, and assume that $\alpha_{1}^{T_1} = \alpha_2^{T_2}$. If $T_1$ has diameter larger than $1$, then $T$ is not integral. In particular, there are no balanced integral trees of odd diameter larger than three. 
\end{corollary}
\begin{proof}
	Our goal is to apply the even length case of Theorem~\ref{thm:main} (in a slightly stronger statement, that follows straight from the proof). Note that hypothesis (1) holds, and hypotheses of type (2) hold because
	\[
		\alpha_1^T = \alpha_1^{T_1} - \frac{1}{\alpha_2^{T_2}},
	\]
	so $\alpha_1^T(\theta) = 0$ implies that $\alpha_1^{T_1}(\theta) \neq \infty$ (and $\alpha_1^{T_1} =\alpha_2^{T_2}$). 
	
	Recall that $\alpha_1^{G^\pm} = \alpha_1^{T_1} \mp 1$, thus we set $\alpha =  \alpha_1^{T_1}$ and $E = (-1,1)$, and apply the argument in the final part of the proof of the even case in Theorem~\ref{thm:main}. As $T_1$ has diameter larger than $1$, it follows that $\alpha_{1}^{T_1}$ has at least three zeros, therefore $\alpha^{-1}(E)$ contains at least three intervals. If $\theta_1$ and $\theta_1'$ are the largest zeros of $\alpha_1^{G^\pm}$ respectively, $\theta_2$ and $\theta_2'$ are the smallest, it follows that
	\[
		2 = \mu(E) > (\theta_1 - \theta_1') + (\theta_2 - \theta_2'),
	\]
	and so either $\theta_1' \in (\theta_1-1,\theta_1)$ or $\theta_2' \in (\theta_2-1,\theta_2)$, therefore some eigenvalue of $T$ is not an integer.
\end{proof}

We now show another application of our technology to the study of integral trees (integral graphs, in fact). This time symmetry will not be a requirement.

Jacobs and Trevisan \cite{TrevisanJacobsOriginal} introduced an algorithm to locate eigenvalues of trees. It was extended to several graph families, resulting in the book \cite{TrevisanJacobsBook}, and also used to derive an important result on the distribution of Laplacian eigenvalues of trees \cite{TrevisanJacobsMostLaplacian}. The rational functions introduced earlier allow for a slightly alternative interpretation of this algorithm, which is quite convenient to deal with rooted products on paths. We will briefly develop this interpretation, followed by an interesting application to study integral trees.

 Assume $T$ is a rooted tree (as in having a root vertex) on vertices $\{1,...,n\}$, and make vertex $1$ the root. For any vertex $i$ of $T$, let $T(i)$ denote the downward tree obtained from $T$ upon deleting all vertices which are not equal to $i$ or to some descendant of it. We define the rational function on the variable $x$ by 
 \begin{equation}\label{eq:algorithm}
 	d_i = x - \sum_{j \text{ child of }i} \frac{1}{d_j}.
 \end{equation}
 It is immediate to verify from Equation~\eqref{eq:schwenk} that
 \begin{equation}
 	d_i = \alpha_{i}^{T(i)}. \label{eq:algorithmalpha}
 \end{equation}
 Recall now that $d_i$ will be always increasing, except on its poles, and that it has simples zeros and simple poles.
 
 At $x = \theta$, the rational function $d_i$ can be a negative real number, equal to zero, a positive real number, or have a pole. We denote these facts by, respectively, $d_i(\theta) < 0$, $d_i(\theta) = 0$, $d_i(\theta) > 0$ or $d_i(\theta) = \infty$.
 
 The following result is morally equivalent to Theorem 3 in \cite{TrevisanJacobsOriginal}.
 
 \begin{theorem}
 	For $x = \theta$, the number of vertices $i$ so that $d_i(\theta) > 0$ or $d_i(\theta) = \infty$ is equal to the number of eigenvalues of $T$ in the interval $(-\infty,\theta)$, counted with multiplicity.
 \end{theorem}
\begin{proof}
	The first thing to note is that as $\theta$ increases, the number of positive signs and poles never decreases. This is obviously true if $d_i(\theta) > 0$ and $i$ is a leaf, because in this case Equation~\ref{eq:algorithmalpha} gives that
	\[
		d_i(x) = x.
	\]
	If $i$ is not a leaf, then the only reason $d_i$ would no longer be positive or a pole after $\theta$ is if $\theta$ is pole. In this case, Equation~\ref{eq:algorithm} implies that $d_j(\theta) = 0$ for some $j$ child of $i$, and thus the number of nodes with positive signs or poles at $\theta+\epsilon$ does not decrease.
	
	This last paragraph also explains that positive signs within the tree only appear when the vertex is possibly the root or it is not and its parent ceases to be positive. So there are only two ways the number of positive signs or poles increases: in the first, the root becomes positive right after $\theta$, and this occurs when
	\[
		d_1(\theta) = \alpha_{1}^{T(1)}(\theta) = \alpha_{1}^{T}(\theta) = 0,
	\]
	and this is equivalent to $\theta$ being an eigenvalue of $T$ in the support of $1$. If $\theta$ has larger multiplicity in the graph, then all of its other eigenvectors are also eigenvectors of $T\setminus 1$ and will correspond to positive signs created in the next way. The second way is when a vertex $i$ finds a pole at $\theta$ because two or more of its children become $0$ at $\theta$. In this case, note that $\theta$ will become an eigenvalue of $T(j)$ in the support of $j$ for any $j$ child of $i$ that became 0, and so if $m$ children become $0$ (and positive right after), note that there will be an eigenspace of dimension $m-1$ for $\theta$ in $T$, all of which are $0$ everywhere outside of $T(i) \setminus i$.
\end{proof}

Note that the theorem could have been stated alternatively by saying that for $x = \theta$, the number of vertices $i$ so that $d_i(\theta) \geq 0$ is equal to the number of eigenvalues of $T$ in the interval $(-\infty,\theta]$, counted with multiplicity.

We assume now our tree $T$ has the following format:
\begin{figure}[H]
\begin{center}
	\includegraphics{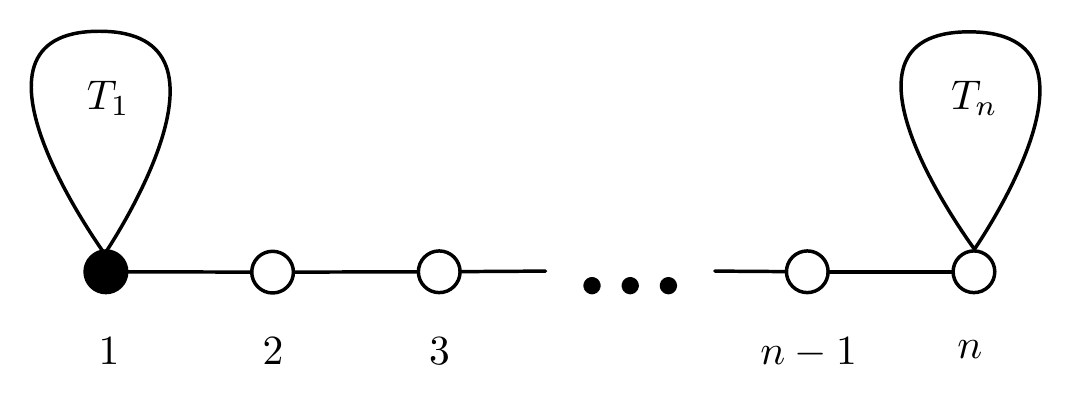}
	\caption{A tree that contains two vertices separated by a subdivided edge.} \label{fig:treepath}
\end{center}
\end{figure}

Assume $n \geq 8$. Our goal is to show that there are too many distinct eigenvalues in the interval $(-2,2)$. So many that at least one is not an integer. Assume $1$ is the root.

\begin{itemize}
    
    
    \item Make $\theta = 2$. If $i>1$, then
    	\[d_i(2) = 2-\frac{1}{d_{i+1}(2)},\] 
    So once there is a negative value or a zero on the path, all the remaining values going towards the root become positive or poles, as $2-1/x$ maps $[1,\infty]$ to itself. Therefore, the number of positive values or poles in the path is at least $n-2$ (at most one negative or one zero, and we cannot control what happens at vertex $1$).
    
    \item Make $\theta = -2$. If $i>1$, then
    	\[d_i(-2) = -2-\frac{1}{d_{i+1}(-2)},\] 
    So once there is a positive value or a zero on the path, all the remaining values going towards the root become negative or poles. Therefore, the number of positive values or zeros is at most $2$.
\end{itemize}

\begin{theorem}
	If a graph $G$ contains two vertices $1$ and $n$ for which there is a unique path of length at least $7$ between them, and all $n-2$ inner vertices of the path have degree $2$, then the graph has at least one eigenvalue that is not an integer.
\end{theorem}

\begin{proof}
	We already did half the work. It remains to argue that in Figure~\ref{fig:treepath} the trees $T_1$ and $T_n$ can be arbitrary graphs instead, and that the sign count preceding the statement of the theorem is indeed counting new distinct eigenvalues.
	
	In the definition of $d_i$ in Equation~\eqref{eq:algorithm}, we could have considered that on each vertex $i$ of $T$ there is a rooted graph $G_i$, and make $d_i = x - \sum_{j} d_j^{-1} - (\alpha_{i}^{G_i})^{-1}$. Equation~\ref{eq:algorithmalpha} still holds.
	
	Now look at $G$, and number the vertices in the path between $1$ and $n$. Say $j \in \{2,...,n\}$, and assume that $d_j(\theta) = 0$ for some $\theta \in (-2,2)$. Then $d_{j-1}(\theta) = \infty$, and so, choosing $\epsilon$ sufficiently small, we have $\theta\pm\epsilon \in (-2,2)$; and $d_{j-1}(\theta-\epsilon) > 0$, $d_{j-1}(\theta+\epsilon) < 0$, $d_j(\theta-\epsilon) < 0$ and $d_j(\theta+\epsilon) > 0$. Thus the number of values $\geq 0$ within the path can only increase when the root becomes equal to $0$. According to our count, this must happen at least four times between in $(-2,2)$, and because $d_1(\theta) = \alpha_1^G(\theta)$, which has simple zeros corresponding to the eigenvalues of $G$ in the support of $1$, it follows that there are at least four distinct eigenvalues of $G$ in $(-2,2)$, so at least one of them is not an integer.
\end{proof}

As a final remark, recall that Schwenk showed in \cite{schwenk1973almost} that almost all trees contain any given limb, so by making this limb equal to $P_8$ (we may as well call it a tail), we are observing the unsurprising fact that, despite there being infinitely many integral trees of arbitrarily large diameters, almost all trees are not integral.

\section{Future research} \label{sec:problems}

This paper was heavily motivated by the conjecture in \cite{CoutinhoLiu2} that no tree on 4 or more vertices admits perfect state transfer. A forthcoming article by the same authors will address this conjecture.

Lippner, Kempton and Yau \cite{KemptonLippnerYauPotential} showed that if the graphs $G_i$ in Figure~\ref{fig:rooted}\footnote{This is the last time we refer back to this figure, so you can forget it now.} are weighted loops instead, then perfect state transfer does not occur between vertices $1$ and $n$. It would be really interesting to find a version our result that allows for the $G_i$ to be weighted, leading to a simultaneous generalization of their and our work in ruling out perfect state transfer, or to find a counter-example.

Integral trees are arguably one of the most studied classes of graphs in spectral graph theory. We have made humble progress towards a classification of this class, but we believe our methods have not yet been fully exploited. In particular, we have not tried much to find more examples of subtrees that sit between two vertices and force some eigenvalue of the tree to be non-integral.

\subsection*{Acknowledgements}

Authors acknowledge the funding from FAPEMIG that supported this research. Gabriel Coutinho acknowledges the support of CNPq.

\bibliographystyle{plain}
\IfFileExists{references.bib}
{\bibliography{references.bib}}
{\bibliography{../references}}

	
\end{document}